\documentclass{birkjour}
\usepackage[utf8]{inputenc}
\usepackage{amsmath}
\usepackage{calc}
\usepackage[dvipsnames]{xcolor}
\usepackage{accents}
\usepackage[all]{xy}   
\usepackage{tikz-cd}
\tikzcdset{arrow style=tikz, diagrams={>=stealth}}                        %
  
\CompileMatrices                            

\UseTips                                    

\input xypic
\usepackage[bookmarks=true]{hyperref}       

\usepackage{amssymb,latexsym,amsmath,amscd}
\usepackage{xspace}
\usepackage{color}
\usepackage{graphicx}
\usepackage{dsfont}
\usepackage{cancel}

\reversemarginpar

\theoremstyle{plain}
\newtheorem{theorem}{Theorem}[section]
\newtheorem*{theorem*}{Theorem}

\newtheorem{lemma}[theorem]{Lemma}
\newtheorem{conjecture}[theorem]{Conjecture}

\theoremstyle{definition}

\newtheorem{remark}[theorem]{Remark}
\newtheorem{example}[theorem]{Example}

\newtheorem{question}[theorem]{Question}

\newcommand{\enm}[1]{\ensuremath{#1}}          %
\newcommand{\op}[1]{\operatorname{#1}}
\newcommand{\cal}[1]{\mathcal{#1}}

\renewcommand{\AA}{\enm{\mathbb{A}}}
\newcommand{\PP}{\enm{\mathbb{P}}}

\newcommand{\Cc}{\enm{\cal{C}}}
\newcommand{\Dd}{\enm{\cal{D}}}
\newcommand{\Ee}{\enm{\cal{E}}}
\newcommand{\Ff}{\enm{\cal{F}}}
\newcommand{\Gg}{\enm{\cal{G}}}
\newcommand{\Hh}{\enm{\cal{H}}}

\newcommand{\Ll}{\enm{\cal{L}}}
\newcommand{\Mm}{\enm{\cal{M}}}

\newcommand{\Oo}{\enm{\cal{O}}}

\newcommand{\Xx}{\enm{\cal{X}}}

\renewcommand{\phi}{\varphi}
\renewcommand{\theta}{\vartheta}
\renewcommand{\epsilon}{\varepsilon}

\newcommand{\Aut}{\op{Aut}}

\newcommand{\codim}{\op{codim}}

\renewcommand{\to}[1][]{\xrightarrow{\ #1\ }}

\renewcommand{\sl}{\textsl{sl}}

\newcommand{\old}[1]{}
\newcommand{\vni}{\vskip 4pt \noindent}
\newcommand{\comu}{\codim_{\Mm_g}\mu{(\Hh)}}

\newcommand{\h}{\ensuremath{\mathcal}}
\newcommand{\HL}{\ensuremath{\mathcal{H}^\mathcal{L}_}}
\newcommand{\HO}{\ensuremath{\mathcal{H}^\mathcal{}_}}

\begin{document}

\title[Hilbert scheme of curves of unexpected dimension,  moduli number
]
{Hilbert scheme of 
smooth projective curves
of unexpected  dimension \& existence of a component with less than the expected number of moduli
}


\author[Changho Keem]{Changho Keem}
\address{
Department of Mathematics,
Seoul National University\\
Seoul 151-742,  
South Korea}

\email{ckeem@snu.ac.kr \& ckeem1@gmail.com}





\date{\today}
\maketitle
\begin{abstract}
We denote by $\mathcal{H}_{d,g,r}$ the Hilbert scheme of smooth curves of degree $d$ and genus $g$ in $\mathbb{P}^r$. Denoting by $\mathcal{M}_g$ the moduli space of smooth curves of genus $g$, let $\mu: \mathcal{H}_{d,g,r}\dasharrow \mathcal{M}_g$ be the natural map sending $X\in\mathcal{H}_{d,g,r}$ to its isomorphism class $\mu (X)=[X]\in\mathcal{M}_g$.

It has been conjectured that a component $\mathcal{H}\subset\mathcal{H}_{d,g,r}$ has the minimal possible dimension $$\Xx(d,g,r):=3g-3+\rho(d,g,r)+\dim\operatorname{Aut}(\mathbb{P}^r)$$ \noindent if $\codim_{\mathcal{M}_g}\mu(\mathcal{H})\le g-5$ provided $\Xx(d,g,r)\ge 0$, where $\rho(d,g,r):=g-(r+1)(g-d+r)$ is the Brill-Noether number. In this article, we exhibit examples against the conjecture discuss further for the study of the functorial map $\mu: \mathcal{H}{d,g,r}\dasharrow\mathcal{M}_g$ along this line.

A component $\mathcal{H}\subset \mathcal{H}_{d,g,r}$ is said to have the {\it expected number of moduli} if $$\dim\mu({\Hh})=\min\{3g-3, 3g-3+\rho(d,g,r)\},$$provided $3g-3+\rho(d,g,r)\ge 0$. The existence of a component with strictly less than the expected number of moduli has not been known. In this paper, we show the existence of components with less than the expected number of moduli.\end{abstract}
\section{\quad An overview, preliminaries and basic set-up}


Given non-negative integers $d$, $g$ and $r\ge 3$, let $\mathcal{H}_{d,g,r}$ be the Hilbert scheme of smooth projective curves,
which is the union of components whose general point corresponds to a smooth irreducible and non-degenerate curve of degree $d$ and genus $g$ in $\PP^r$.
Denoting the moduli space of smooth curves of genus $g$ by $\Mm_g$,
 let  $$\mu: \HO{d,g,r}\dasharrow \Mm_g$$ be the natural rational map - which we call the {\it moduli map} - sending $X\in\HO{d,g,r}$ to its isomorphism class
$\mu (X)=[X]\in\h{M}_g$.

In this paper, we are interested in the relationship between 
$\codim_{\Mm_g}\mu(\Hh)$ and $\dim\Hh$.  For any irreducible component $\Hh\subset\HO{d,g,r}$, it is known that 
\begin{equation}\label{md}
\dim\Hh\ge\Xx(d,g,r):=\dim\Mm_g+\rho(d,g,r)+\dim\Aut(\PP^r).
\end{equation}

\noindent
In the {\it Brill-Noether range}, i.e.  when  $$\rho (d,g,r):=g-(r+1)(g-d+r)\ge 0,$$
there exits a {\it unique} component $$\Hh_0\subset\HO{d,g,r}$$
dominating $\Mm_g$ \cite[Theorem 19.9, page 348]{EH2}. Furthermore,  this component $\Hh_0$ has the {\it expected dimension}
$$\dim\Hh_0=\Xx(d,g,r).$$ On the other hand, still under the assumption $\rho(d,g,r)\ge 0$,  it is not totally clear
if $\Hh_0$ is the only component having the minimal possible (and expected) dimension. A priori there may exist a component $\Hh\subset\HO{d,g,r}$
with $\dim\Hh=\Xx(d,g,r)$ which does not dominate $\Mm_g$. If this occurs, this would lead to the reducibility of $\Hh_{d,g,r}$ as was seen in \cite{Keem}.
At this point, one may raise the following naive but natural questions. 

\begin{itemize}
\item[(A)]
Suppose $\rho(d,g,r)\ge 0$. Identify all (or as many as possible) the components $\Hh\subset\Hh_{d,g,r}$ with $\dim\Hh=\Xx(d,g,r)$. 

\item[(B)] Without the assumption  $\rho(d,g,r)\ge 0$, is it possible 
to make an estimate (or upper bound) for 
$\codim_{\Mm_g}\mu{(\Hh)}$ -- hopefully linearly in $g$ --
such that $\dim\Hh=\Xx(d,g,r)$?
\end{itemize}

Regarding  question (A), it is merely known that there exists a component $\Hh\subset \Hh_{d,g,r}$ other than $\Hh_0$ with $\dim\Hh=\Xx(d,g,r)$; cf. \cite[Theorem 3.4]{Keem}, \cite[Propositions 3.3 \& 3.5]{CKP} or Example \ref{trigonal} below. We will reproduce one typical type  of this kind in 
 Example \ref{trigonal}.
 However, all such known examples occur as a component whose general elements are {\it non-linearly normal} curves under the 
assumption $\rho(d,g,r)\ge 0$. We will be more specific with this phenomena in the next section; Question \ref{linearlynormal}.

Outside the Brill-Noether range $\rho(d,g,r)<0$, no component $\Hh\subset\HO{d,g,r}$ dominates $\Mm_g$ by the
Brill-Noether theorem. 
However, one may still ask {\it under what numerical or geometrical condition, does there exist a component $\Hh$ 
with $\dim\Hh=\Xx(d,g,r)$?} 
\, Loosely  speaking, one expects that if $\codim_{\Mm_g}\mu(\Hh)$ gets smaller,  curves in $\Hh$ tend 
to be nearer to a general curve and therefore it is likely that they would behave like curves
with general moduli.  Thus  it is natural to expect that 
the estimate $\dim\Hh=\Xx(d,g,r)$ may continue 
to hold for a component consisting curves that are not too special.
Along the line of these thoughts, the following has been suggested \cite[Conjecture, pp142--143]{H0}.

\begin{itemize}
\item[(C)] Find a function $\beta (g)$ (probably linear in $g$), which ensures the existence of a component
$\Hh\subset\HO{d,g,r}$ of the expected dimension in case $$\codim_{\Mm_g}\mu(\Hh)\le \beta(g).$$
\end{itemize}
Specifically  one finds the following updated conjecture which has appeared recently \cite[Conjecture 19.6, page 344]{EH2}.
\begin{conjecture}\label{con} If $\codim_{\Mm_g}\mu(\Hh)\le \beta(g)=g-5$, then $\dim\Hh=\Xx(d,g,r)$.
\end{conjecture}

\vni
$\blacklozenge$
In this article, we exhibit several examples of triples $(d,g,r)$ such that  Conjecture \ref{con}
is no longer valid; Examples \ref{first}, \ref{extensive}. We also make a further discussion for the study of the moduli map along this line; Remark \ref{discussion}.








\vni
$\blacklozenge$ We  study a related problem regarding the existence of a component having less than the expected number of moduli; Example \ref{kgonal}, Remark \ref{discussion3}, Question \ref{discussion4}

\vni

We say that a component $\Hh\subset\HO{d,g,r}$ has the expected number of moduli 
if $$\dim\mu(\Hh)=\min\{3g-3,3g-3+\rho(d,g,r)\}.$$
under the condition $$\lambda(d,g,r):=3g-3+\rho(d,g,r)\ge0.$$
When $\rho(d,g,r)\ge 0$, the principal component $\Hh_0$ is {\it the only} component having the 
expected number of moduli. Other components (if any) have {\it less} than the expected number of moduli.
Thus, we are interested in the existence of a component {\it with less than the expected number of moduli} only when 
$\rho(d,g,r)<0$, i.e. outside the Brill-Noether range.
We exhibit a class of  curves forming a component of a Hilbert scheme with less than the expected number of moduli
in the final section.

\section{Preliminaries, Notations and Some examples}

\subsection{Preliminaries, Notations and Conventions}
For notation and conventions, we  follow those in \cite{ACGH} and \cite{ACGH2}; e.g. $\pi (d,r)$ is the maximal possible arithmetic genus of an irreducible,  non-degenerate and reduced curve of degree $d$ in $\PP^r$.
Following classical terminology, a linear series of degree $d$ and dimension $r$ on a smooth curve is denoted by $g^r_d$.
A  linear series $\Dd=g^r_d$ ($r\ge 2$) on a smooth curve $C$ is called {\it birationally very ample} when the morphism 
$C \rightarrow \mathbb{P}^r$ induced by  the moving part of $\Dd$ is birational onto its image.
A  linear series $g^r_d$  is said to be {\it compounded} if the morphism induced by  $g^r_d$ gives rise to a non-trivial covering map $C\rightarrow C'$ of degree $k\ge 2$. Given a smooth non-degenerate curve $C\subset\PP^r$, 
the {\it residual curve of $C$} -- sometimes denoted by  $C^\vee$ -- is the image curve of the morphism induced by the moving part of $|K_C(-1)|$. 
For a complete linear series $\h{D}$ on a smooth curve $C$, the residual series $|K_C-\h{D}|$ is denoted by $\h{D}^\vee$.


We use the  following basic frameworks of our study borrowed from \cite{ACGH2}  and \cite[\S1  \& \S 2]{AC2}. Given an isomorphism class $[C] \in \mathcal{M}_g$ corresponding to a smooth irreducible curve $C$, there exist a neighborhood $U\subset \mathcal{M}_g$ of the class $[C]$ and a smooth connected variety $\mathcal{M}$ which is a finite ramified covering $h:\mathcal{M} \to U$, as well as  varieties $\mathcal{C}$, $\mathcal{W}^r_d$ \& $\mathcal{G}^r_d$ proper over $\mathcal{M}$ with the following properties:
\begin{enumerate}
\item[(1)] $\xi:\mathcal{C}\to\mathcal{M}$ is a universal curve, i.e. for every $p\in \mathcal{M}$, $\xi^{-1}(p)$ is a smooth curve of genus $g$ whose isomorphism class is $h(p)$,
\item[(2)] $\mathcal{W}^r_d$ parametrizes the pairs $(p,L)$ where $L$ is a line bundle of degree $d$ and $h^0(L) \ge r+1$ on $\xi^{-1}(p)$,
\item[(3)] $\mathcal{G}^r_d$ parametrizes the pairs $(p, \mathcal{D})$, where $\mathcal{D}$ is possibly an incomplete $g^r_d$ on $\xi^{-1}(p)$.
\end{enumerate}
Let $\widetilde{\mathcal{G}}$ ($\widetilde{\mathcal{G}}_\mathcal{L}$,  resp.) be  the union of components of $\mathcal{G}^{r}_{d}$ whose general element $(p,\mathcal{D})$ corresponds to a very ample (very ample and complete, resp.) linear series $\mathcal{D}$ on the curve $C=\xi^{-1}(p)$. Recalling that an open subset of $\mathcal{H}_{d,g,r}$ consisting of smooth irreducible and non-degenerate curves is a $\Aut(\PP^r)$-bundle over an open subset of $\widetilde{\mathcal{G}}$, the irreducibility of $\widetilde{\mathcal{G}}$ guarantees the irreducibility of $\mathcal{H}_{d,g,r}$. Likewise, the irreducibility of $\widetilde{\mathcal{G}}_\mathcal{L}$ ensures the irreducibility of 
 $\mathcal{H}_{d,g,r}^\mathcal{L}$, the union of components of $\HO{d,g,r}$ whose general element is linearly normal.
\vni
Given an irreducible family $\h{F}\subset\h{G}^r_d$ whose general member is complete,  the closure of the family $\{\h{D}^\vee| \h{D}\in \h{F}, ~\h{D} \textrm{ is complete}\}\subset \h{W}^{g-d+r-1}_{2g-2-d}$ is  sometimes denoted by $\h{F}^\vee$.
Given $r\ge s$, $\mathbb{G}(s,r)$ denotes the Grassmannian of dimension $s$ linear subspaces of $\PP^r$. 
Throughout we work 
over the field of complex numbers.

\subsection{Some old and new Examples}
We start with the following example of a {\it reducible} Hilbert scheme with $\rho(d,g,r)\ge 0$, which is a simplified version of \cite[Theorem 2.3]{Keem}. This example been known for some years, which we choose to present here with a brief outline in order to provide readers a motivation and a flavor with an emphasis on what we shall discuss in this paper.

\begin{example}\label{trigonal} (i) If $11\le g\le 16$, $\HO{2g-8,g,g-8}$ is non-empty and irreducible whose  general element is linearly normal. 

\noindent
(ii) If $g\ge 17$, there exist an extra component $\Hh\subset\HO{2g-8,g, g-8}$, $\Hh\neq\Hh_0$ with $$\comu=g-4$$ 
\[
\dim\Hh=\begin{cases}\Xx(2g-8,g,g-8) \hskip 51pt \text{ if } g=17\\
\Xx(2g-8,g,g-8)+(g-17), \text{ if } g>17.
\end{cases}\]
A general element of $\Hh$ is a trigonal curve, which is {\it not linearly normal} in $\PP^{g-8}$. 
\end{example}
\begin{proof} Since $\rho(2g-8,g,g-8)\ge0$, there is a unique component $\Hh_0$ dominating $\Mm_g$.
~Assume the existence of a 
component $\Hh\neq\Hh_0$. Recall that $\Hh$ consists of curves $C$ such that $\Oo_C(1)$ is special \cite[Chapter 2]{H1}, \cite[Proposition 1.1]{KKL}. 

Take a general $C\in\Hh$ with the hyperplane series $\Dd=g^{g-8}_{2g-8}\subset|\Oo_C(1)|$. Set $\alpha=h^1(C, \Oo_C(1))>0$. By Clifford's theorem, $1\le \alpha\le 4$. If $\alpha =4$, $|K_C(-1)|=g^3_6$ and $C$ is a hyperelliptic curve, which does carry a birationally very ample special linear series. If $1\le\alpha\le 2$, then by  H. Martens' theorem for the Brill-Noether locus $W^{\alpha-1}_6(C)$ consisting of special linear series on a fixed curve \cite[Theorem 5.16]{EH2} and then passing to the residual series' 
\begin{align*}
\dim\Hh&\le \dim\Aut(\PP^{g-8})+\dim\mathbb{G}(g-8,g+\alpha -8)+\dim W^{g+\alpha -8}_{2g-8}(C)+\dim\mu(\Hh)\\
& \le\dim\Aut(\PP^{g-8})+\dim\mathbb{G}(g-8,g+\alpha -8)+\dim W^{\alpha-1}_6(C)+3g-4\\
& \le\dim\Aut(\PP^{g-8})+\dim\mathbb{G}(g-8,g+\alpha -8)+(6-2(\alpha-1))+3g-4\\
& <\dim\Aut(\PP^{g-8})+\lambda(2g-8,g,g-8)=\Xx(2g-8,g,g-8),
\end{align*}
contradicting \eqref{md} if $g\ge 11$. 
Thus we are  left with $\alpha=3$, $\Ee:=|K_C(-1)|=g^2_6$ and $\Ee$ is compounded (since $g\ge 11$) inducing  $\eta: C\to E\subset\PP^2$. $C$ is not
bielliptic; if so, then $|K_C-\eta^*(g^2_3)|$ is not very ample; cf. Lemma \ref{easylemma}.
 Thus $C$ is trigonal, 
$\Ee=2g^1_3$ and the hyperplane series 
$$\Dd\subset |K_C(-2g^1_3)|=g^{g-5}_{2g-8}$$
is a (incomplete) subseries of the unique very ample $g^{g-5}_{2g-8}=|K_C(-2g^1_3)|$. 
It follows that there is a component $\Hh\neq \Hh_0$ over $\Gg\subset\Gg^{g-8}_{2g-8}$ consisting of incomplete very ample subseries' $g^{g-8}_{2g-8}\subset |K_C(-2g^1_3)|$ on a trigonal curve $C\in\Mm^1_{g,3}$. 
Thus $\Hh$ dominates the irreducible locus $\Mm^1_{g,3}\subset\Mm_g$ consisting of trigonal curves and 
\begin{align*}
\dim\Hh&=\dim\Aut(\PP^{g-8})+\dim\mathbb{G}(g-8,g-5)+\dim\Mm^1_{g,3}\\&=\dim\Aut(\PP^{g-8})+3(g-7)+2g+1\\
&
\ge 
\Xx(2g-8,g,g-8)=\dim\Aut(\PP^{g-8})+4g-3
\end{align*}
if $g\ge 17$. 
Therefore, we conclude

\begin{itemize}
\item[(i)]
$\HO{2g-8,g,g-8}$ is non-empty and irreducible with the only component $\Hh_0$ (the principal component)  if $11\le g\le16$. A general element is non-special and 
complete; \cite[Ch. IV, Proposition 6.1]{Hartshorne}.
\item[(ii)]
$\HO{2g-8,g,g-8}$ is reducible with two components if $g\ge 17$; the only extra component $\Hh$ dominates the locus
$\Mm^1_{g,3}$ of trigonal curves and hence $$\codim_{\Mm_g}\mu(\Hh)=g-4.$$ 
Every member $C\in\Hh$ is a {\it non linearly normal} curve \&
\[
\dim\Hh=\begin{cases}\Xx(2g-8,g,g-8) \hskip 51pt \text{ if } g=17\\
\Xx(2g-8,g,g-8)+(g-17), \text{ if } g>17.
\end{cases}\]
\end{itemize}
\end{proof}
\begin{remark} 
There abound examples of this kind which can be found in \cite[Theorem 3.4]{Keem}, \cite[Proposition 3.5]{CKP}. All such known examples
have at least one extra component consisting of non linearly normal curves. Thus, we may pose 
the following rather {\it crude} question. 
\end{remark}
\begin{question}\label{linearlynormal}
If  $\rho(d,g,r)\ge 0$, does there exist a Hilbert scheme of smooth projective curves with a component $\Hh\subset\HO{d,g,r}$ such that $\Hh\neq\Hh_0$ \& $\Hh$ generically consists of linearly normal curves ?

\end{question}
Unfortunately, the author does not know of a reasonable answer to the above question.
 As a matter of fact, some people (e.g., \cite[page 489]{CS} or [MathSciNet; AMS Mathematical Reviews MR1221726(95a:14026)]) believe that this is what Severi had in mind in the 
two papers where he claimed the irreducibility of $\HO{d,g,r}$ in the Brill-Noether range $\rho(d,g,r)\ge 0$ with incomplete proofs; \cite{Sev1}, \cite{Sev}. 

For the rest of this paper, we are solely interest in and deal with triples $(d,g,r)$ {\it outside the Brill-Noether range}, i.e. $\rho(d,g,r)<0$ unless otherwise specified.
The next example -- which is of a similar sort to Example \ref{trigonal} -- indicates that Conjecture \ref{con} {\it cannot} be revised into the one with the 
relaxed condition $$\codim_{\Mm_g}\mu(\Hh)\le g-4,$$ while imposing extra {\it but seemingly natural} condition {\it a general element in $\Hh$ being linearly normal. }
\begin{example}
(i) For $(d,g,r)=(15,16,5)$, $\HO{d,g,r}=\HL{d,g,r}$ is reducible with three components  \cite[Proposition 5.1]{rmi}. 
There is a component $\Hh$ among  three  with 
$$\codim_{\Mm_g}\mu(\Hh)=g-4 \text{ dominating } \Mm^1_{g,3}, \,\dim\Hh=68>\Xx(d,g,r)=60.$$ It can be easily shown that 
a complete very ample $\Dd=g^5_{15}$ on a curve of genus $g=16$ such that $\Dd^\vee=g^5_{15}$ is compounded is of the form $\Dd=|K_C-5g^1_3|$ on a trigonal curve and conversely. 
Obviously a general element of $\Hh$ is linearly normal. 

\vni
(ii)
 For $(d,g,r)=(14,14,5)$ , $\HO{14,14,5}=\HL{14,14,5}$ is irreducible dominating $\Mm^1_{g,3}$ \cite[Proposition 4.3(iii)]{flood}. Indeed, the same reasoning as above applies; a general $\Dd\in\Gg\subset \Gg^5_{14}$ is of the form $\Dd=|K_C-4g^1_3|$ and $$\codim_{\Mm_g}\mu(\Hh)=g-4, \dim\HO{d,g,r}=64>\Xx(14,14,5).$$
 \vni
 (iii) Therefore  $\beta(g) =g-4$ is not a right upper bound of the codimension of the image under $\mu$ of a 
 component $\Hh$ of the expected dimension
 generically consisting of linearly normal curves. 

\end{example}
\vskip 4pt
Before proceeding to the next subsection, we give  examples $\Hh\subset\HO{d,g,r}$ with $\dim\Hh=\Xx(d,g,r)$ 
such that $$\codim_{\Mm_g}\mu(\Hh)\lneq g-4,$$
thereby may possibly provide a {\it positive clue}
for the validity of the Conjecture \ref{con} with {\it smaller} $\beta(g)$. 
\begin{example}
(i)
For $(d,g,r)=(14,12,5)$, $\HL{d,g,r}$ is reducible with two components $\Hh_i, i=1,2$ of the {\it same expected dimension} $\Xx(d,g,r)$; \cite[Proposition 4.3(v)]{flood}, \cite[Theorem 3.4(ii)]{lengthy}. There is component $\Hh_1$ dominating $\Mm^1_{g,4}$ consisting of linearly normal curves embedded by $|K_C-2g^1_4|$ on general $4$-gonal curves, thus $$\codim\mu(\Hh_1)=\codim\Mm^1_{g,4}=g-6<g-5.$$
A general element of the other component $\Hh_2$ has a plane model of degree $8$ with $\delta=9$ nodal singularities.
A usual dimension count yields $$\codim_{\Mm_g}\mu(\Hh_2)=g-6<g-5.$$

\vni
(ii) For $(d,g,r)=(14,11,5)$, $\HO{d,g,r}$ is irreducible of the expected dimension dominating $\Mm^1_{g,6}$ \cite[Theorem 2.3]{BFK}
$$\codim_{\Mm_g}\mu(\HO{14,11,5})=\codim_{\Mm_g}\Mm^1_6=g-10<g-5.$$
\end{example}
\vni
\vni

\subsection{Examples against Conjecture \ref{con} with $\beta(g)=g-5$}
In this subsection, we exhibit  an example $\HO{d,g,r}$  such that  Conjecture \ref{con} with $\beta(g)=g-5$
turns out to be invalid; Example \ref{first}.
Then we will try to make sense out of this example and seek for further generalization; Example \ref{extensive}.

\begin{lemma}\label{easylemma}
{\rm(i)}
Let $C\stackrel{\eta}{\rightarrow} E$ be a double covering of a curve $E$ of genus $h\ge 1$. Let $\h{E}=g^s_{e}$ be a non-special linear series on $E$. Assume that $|\eta^*(g^s_{e})|=g^s_{2e}$. Then the base-point-free part of the complete  
$|K_C(-\eta^*(g^s_{e}))|$ is compounded.
{\rm(ii)}
Let $C$ be a trigonal curve of genus $g\ge 5$. Suppose $\dim|2g^1_3+\Delta|=2$ 
for an effective divisor $\Delta$ with $\deg\Delta>0$. Then $|K_C-2g^1_3-\Delta|$ is not very ample. 
\end{lemma}
\begin{proof}
(i) Note that for any $p\in E$,  $|g^s_{e}+p|=g^{s+1}_{e+1}$ since $g^s_{e}$ is non-special, 
 $$\dim|\eta^*(g^s_{e})+\eta^*(p))|=\dim|\eta^*(g^s_{e}+p)|=\dim|\eta^*(g^{s+1}_{e+1})|\ge \dim\h{E} +1.$$
 Hence $|K_C(-\eta^*(g^s_{e}))|$  is compounded.
 
 \vni
 (ii) Choose $q\in\text {Supp}(\Delta)$ and the trigonal divisor $q+t+s\in g^1_3$ containing $q$. $|2g^1_3+\Delta|^\vee$ is not very ample since 
\begin{align*}
\dim|2g^1_3+\Delta+t+s|
&=\dim| (2g^1_3+(q+t+s)+(\Delta-q)|\\&=\dim|3g^1_3+(\Delta -q)|\ge\dim|2g^1_3+\Delta|+1.
\end{align*}
\end{proof}
\begin{example}\label{first} Fix $r\ge 15$, $g=r+7$, $d=2r+4$. 
Then $\HL{2r+4,r+7,r}$ is reducible with a component $\Hh_{4,0}$ dominating $\Mm^1_{g,4}$,
\begin{align*}
\dim\Hh_{4,0}&=\dim\Mm^1_{g,4}+\dim\Aut(\PP^r)>\Xx(2r+4,r+7,r), \\\codim\mu(\Hh_{4,0})&=\codim\Mm^1_{g,4}=g-6<g-5.
\end{align*}
\end{example}
\begin{proof}
Assume  $\HL{2r+4,r+7,r}\neq\emptyset$. Let $C\in\Hh\subset\HL{2r+4,r+7,r}$ be a general element. Since $g=r+7\ge 22>\binom{7}{2}$, $\Ee:=|K_C(-1)|=g^2_8$ is compounded, inducing a morphism $\eta: C\to E\subset\PP^2$, $h=\text{genus}(E)$.
Note that $\deg\eta\cdot\deg E\le \deg\Ee=8$, $\deg E\ge 2$, thus we have 
 $$(\deg\eta,\deg E, h)\in \{(2,4,3), (2,4,2), (2,3,1), (2,3,0),(3,2,0),(4,2,0)\}.$$
By Lemma \ref{easylemma}, we have either $(\deg\eta, h)=(2,3)$ or $(4,0)$, i.e. $C\in\Xx_{2,3}$ or $C\in\Mm^1_{g,4}$,
where $\Xx_{2,3}\subset\Mm_g$ denotes the locus of double coverings of curves of genus $h=3$. 

\vni
(i) For $C\in\Xx_{2.3}$, $|K_C-\eta^*(\Oo_E(1))|=|K_C-\eta^*(K_E)|=g^r_{2r+4}$ is very ample; cf. \cite[Proposition 4.4]{KM}. Set
$$\Ff_{(2.3)}:=\{|K_C-\eta^*(\Oo_E(1))|; C\in\Xx_{2,3}, C\stackrel{\eta}{\to} E\}\subset\Gg^r_{2r+4}.$$
Recall that $\Xx_{2,3}$ is of pure dimension
$2g-4$, thus $$\dim\Ff_{(2,3)}=2g-4=2r+10> \lambda (2r+4,r+7, r)=r+22.$$
(ii) Let $C$ be general $4$-gonal curve of genus $g$. Note that $|K_C-2g^1_4|=g^r_{2r+4}$
is very ample \cite[Corollary 3.3]{Keem}. We set
$$\Ff_{4,0}:=\{|K_C-2g^1_4|; C\in\Mm^1_{g,4}\}\subset\Gg^r_{2r+4}.$$
 Since there is only one $g^1_4$ on $C$, we have 
 \begin{equation}\label{4gonal}
 \dim\Ff_{4,0}=\dim\Mm^1_{g,4}=2g+3=2r+17\ge\lambda(2r+4,r+7,r).
 \end{equation}
 
 \noindent
Note that  $\dim\Ff_{4,0}>\dim\Ff_{(2,3)}$ whereas $\text{gon}(C_{4,0})<\text{gon}(C_{(2,3)})=6$ ($C_{4,0}\in\Ff_{4,0}$. \&
$C_{(2,3)}\in\Ff_{(2,3)}$)
thus 
 the family $\Ff_{(2,3)}$ is not in the boundary of $\Ff_{4,0}$ by (lower) semi-continuity of gonality from which the reducibility of $\HL{2r+4,r+7,r}$ follows.
 
  In all,  $\HL{2r+4,r+7,r}\neq\emptyset$ has at least two components with one component $\Hh_{4,0}$
 dominating $\Mm^1_{g,4}$, 
\begin{align*}
\dim\Hh_{4,0}&=\dim\Ff_{4,0}+\dim\Aut(\PP^r)>\Xx(2r+4,r+7,r), \\\codim\mu(\Hh_{4,0})&=\codim_{\Mm_g}\Mm^1_{g,4}=g-6<g-5.
\end{align*}
\end{proof}
\begin{remark}\label{discussion} (i) In Example \ref{first}, we assumed $r\ge 15$ in order to avoid complicated situation. For  $3\le r\le14$, $\HL{2r+4,r+7,r}\neq\emptyset$ \& the number of irreducible components and dimension of each component varies depending on $r$ ; cf. \cite[Table 1 ]{lengthy}.

\vni
(ii) It is worthwhile to note that for every $r\ge 5$ there is a component $\Hh_{4,0}\subset\HL{2r+4,r+7,r}$ dominating $\Mm^1_{g,4}$. The equality in \eqref{4gonal} holds only for $r=5$ and strictly inequality for $r\ge 6$. 

\vni
(iii) Thus one may realize that $\codim\mu(\Hh)$ may not the only quantity which governs whether $\dim\Hh=\Xx(d,g,r)$ 
or $\dim\Hh>\Xx(d,g,r)$. More refined upper bound $\beta(g)$ involving other invariants or attributes such as the 
dimension of the ambient projective space needs to be found, if there is any.
\end{remark}

Next we present an example of a Hilbert scheme
 $\HL{d,g,r}$ 
 having several distinct components $\Hh_{k,\delta}\subset\HL{d,g,r}$ 
 dominating $\Mm^1_{g,k}$ such that 
$$\dim\Hh_{k,\delta}>\Xx({d,g,r}), \codim\mu({\Hh_{k,\delta}})\le g-2k+2$$
 {\it for every} $k$'s ($k\ge 4$) in a certain range with respect to $(d,g,r)$.
The main idea involved is the following
\begin{remark} (i) First we consider a general $k$-gonal curve, $k\ge 4$. Then subtract multiple (double) of $g^1_k$ from the canonical series to get a very ample $\Ll=|K_C-mg^1_k|$ as we did in Example \ref{first}.
\vni
(ii) We then make successive projections of the smooth curve $C_\Ll$ induced by $\Ll$ from several $\delta$ general points on $C_\Ll$. 

\vni
(iii) In this way, we arrive at an irreducible family (hopefully a component) $\Hh_{k,\delta}$ dominating $\Mm^1_{g,k}$
whereas $\dim\Hh_{k,\delta}$ 
is big enough,
exceeding $\Xx(d,g,r)$.
\end{remark}
Indeed, this example originated from a study of Hilbert schemes of linearly normal curves with index of speciality $\alpha =g-d+r=3$  without any emphasis or attention on the study of Conjecture \ref{con};  \cite{lengthy}. The main focal point of  \cite{lengthy} was the classical irreducibility problem of Hilbert scheme of linearly normal curves 
with low index of speciality. 
However, we utilize and extract results in \cite{lengthy} adjusted to our current theme and generalize Example \ref{first}. Not only for the convenience of readers but also for maximizing self-containedness of the paper, 
we reproduce
essential  ingredients and steps. 
\begin{lemma}
\cite[Proposition 1.1]{CKM}\label{kveryample} Assume $\rho(k,g,1)=2k-g-2<0$. Let $C$ be a general $k$-gonal curve of genus $g$, $k\ge 2$. Suppose  
\begin{equation}\label{veryamplek}
g\ge 2m+n(k-1)
\end{equation}
for some  $0\le m$, $n\in\mathbb{Z}$. 
 and let $D\in C_m$. Assume that there is no $E\in g^1_k$ with $E\le D.$ Then $\dim|ng^1_k+D|=n$.
\end{lemma}

\begin{lemma}\label{veryample}Let $C$ be a general $k$-gonal curve of genus $g$. We assume 
$$\rho(k,g,1)<0, \,k\ge 4, r\ge 3,  \,\delta:=g-r+1-2k\ge 0, \,g\le 2r+2.$$
For  a general $\Delta\in C_\delta=C_{g-r+1-2k}$, the complete linear series  
$$|K_C-2g^1_k-\Delta|=g^{g-2k+1-\delta}_{2g-2-2k-\delta}=g^r_{g+r-3}$$ is very ample. 
\end{lemma}
\begin{proof} Let $g^1_k$ be a unique pencil of degree $k$ on $C$. Consider the cycle of all effective
divisors of degree $2$ that are subordinate to $g^1_k$, i.e.
$$g^1_{k,2}:=\{D\in C_2: D\le E \textrm{ for some } E\in g^1_k\}. $$
Set 
$$\Sigma:=\{(\Delta, D): \Delta\ge D\}\subset C_\delta\times g^1_{k,2},$$
and let $\pi_1:\Sigma\to C_\delta, \pi_2: \Sigma \to g^1_{k,2}$ be projections.
For $D\in g^1_{k,2}$, $$\pi_2^{-1}(D)\cong C_{\delta-2}, \quad \dim\Sigma =\dim\pi_2^{-1}(D)+\dim g^1_{k,2}=\delta-1.$$
Hence $\pi_1(\Sigma)\subsetneq C_\delta.$ Set 
\begin{equation}\label{sigma}
\overset{o}{C}_\delta:=C_\delta\setminus\pi_1(\Sigma)
\end{equation}
and  take $$\Delta=r_1+\cdots +r_\delta\in \overset{o}{C}_\delta\,$$ i.e., 
no pair $\{ r_i, r_j\}$ lies on a fiber of the unique $k$-sheeted covering of $\PP^1$. 
 For an arbitrary choice $s+t\in C_2$, we take $ D=\Delta +s+t$,  $n=2$ and $m=\delta +2$ in Lemma \ref{kveryample}.  
By the assumption $g\le 2r+2$, we see that the numerical assumption (\ref{veryamplek}) in Lemma \ref{kveryample} ( $g\ge 2m+n(k-1)$) is satisfied.  By our choice of $\Delta$, we have  $\deg\gcd (D,E)\le 3$ for any $E\in g^1_k$ and hence there is no $E\in g^1_k$ such that  $E\le D$ as long as  $k\ge 4$.  Therefore we have  
$$\dim|2g^1_k+\Delta+s+t|=2, $$
for any $s+t\in C_2$ implying $|K_C-2g^1_k-\Delta|$ is very ample. 
\end{proof}
The next example is a simplified variation of \cite[Theorem 3.9]{lengthy}.
\begin{example}\label{extensive}\cite[Theorem 3.9]{lengthy} Fix a triple $(d,g,r)$ such that  $g-d+r=3, r\ge5$ and assume $$r+7\le g\le 2r+2.$$
Then 
\begin{itemize}
\item[(i)]
there is a component $\Hh_{k,\delta}\subset\HL{g+r-3,g,r}$ dominating $\Mm^1_{g,k}$ for every pair $(k,\delta)$ such that  
\begin{equation}\label{cnumber}
4\le k\le \left[\frac{g-r+1}{2}\right] \text{ and }\delta:=g-r+1-2k\ge 0.
\end{equation}
\item[(ii)]
\begin{align}\label{compare}
\dim\Hh_{k,\delta}&=3g-4-r+\dim\Aut(\PP^r)\nonumber\\&\ge 4g-3r-6+\dim\Aut(\PP^r)=\Xx(d,g,r),\\ 
\codim\mu(\Hh_{k,\delta})&=g+2-2k\le g-6.\label{kk}
\end{align}
\item[(iii)] Equality holds in \eqref{compare} if and only if $g=2r+2$. 
\item[(iv)]The number of pairs 
$(k,\delta)$ satisfying \eqref{cnumber} is 
\[
\left[\frac{g-r+1}{2}\right]-3=\begin{cases}
1; \hskip 12pt g=r+7, (k,\delta)=(4,0)\\
1; \hskip 12pt g=r+8, (k,\delta)=(4,1)\\
\ge 2; ~r+9\le g\le 2r+2.
\end{cases}
\]
\end{itemize}
\end{example}
\begin{proof} We note $$\lambda (d,g,r)=3g-3+\rho(d,g,r)=\lambda(g+r-3,g,r)=4g-3r-6.$$
By the condition $g\ge r+7$, 
$$2g-2-d=2g-2-(g+r-3)=g-r+1\ge 8.$$ For any $k\ge 4$ such that 
$g-r+1\ge 2k$, set $\delta=g-r+1-2k\ge 0$. 

We  consider the relative version of the set up -- the one we did in the beginning of the proof of Lemma \ref{veryample} -- over moving $k$-gonal curves. Let $\overset{o}{\Mm}\subset\Mm^1_{g,k}\subset\Mm_g$ be the irreducible family of automorphism free $k$-gonal curves. Let $\Cc$ be the universal curve over $\overset{o}{\Mm}$ and let $\overset{o}{\Cc}_\delta\subset\Cc_\delta$ be the open subset such that the fiber of ~~$\overset{o}{\Cc}_\delta\to\overset{o}{\Mm}$ over $C\in \overset{o}{\Mm}$ is $\overset{o}{C}_\delta:=C_\delta\setminus\pi_1(\Sigma)$; cf. \eqref{sigma}. 
For each $(k,\delta)$ satisfying \eqref{cnumber},  consider the two injective morphisms 
$$\begin{aligned}
&&&&&\Gg^1_k\underset{\overset{o}{\Mm}}{\times}\overset{o}{\Cc}_\delta\hskip 24pt&\stackrel{\zeta}{\to}&&&&\Gg^2_{2k+\delta}=\Gg^2_{g-r+1}&&\stackrel{\kappa}{\to}&&&{\Gg^2_{g-r+1}}^\vee\subset\Gg^r_{g+r-3}\\ 
\\
&&&&&\hskip5pt(g^1_k ,\Delta)
&\mapsto  &&&& |2g^1_k+\Delta|\hskip12pt&&\mapsto &&&|K_C-2g^1_k-\Delta|.
\end{aligned}$$
Let $\iota:=\kappa\circ\zeta$ and
$$\h{F}_{k,\delta}:=\iota(\Gg^1_k\underset{\overset{o}{\Mm}}{\times}\overset{o}{\Cc}_\delta)\subset\h{G}^r_{g+r-3}$$ be the image of  \,$\iota$.
The following inequality holds by the assumption $g\le 2r+2$;
\begin{equation*}\label{kcom}
\begin{split}\lambda(g+r-3,g,r)&=4g-3r-6\le\dim\h{F}_{k,\delta}\\&=
\dim\Gg^1_k\underset{\overset{o}{\Mm}}{\times}\overset{o}{\Cc}_\delta
\\&=\dim\h{G}^1_{k}+\delta\\&=3g-3+\rho(k,g,1)+\delta\\
&=3g-3+(2k-g-2)+(g-r+1-2k)\\
&=3g-4-r.
\end{split}
\end{equation*}
Each element in the irreducible locus $\h{F}_{k,\delta}$  (dominating $\Mm^1_{g,k}$) is very ample by Lemma \ref{veryample}. Let $\Gg$ be a component such that $$\h{F}_{k,\delta}\subset\Gg\subset{\tilde\Gg}_{\Ll}\subset\Gg^r_{g+r-3}.$$ We argue that $\h{F}_{k,\delta}$ is dense in $\Gg$. For this, we distinguish the following three cases.

\vni
(a)
 The  residual series $\h{E}=g^2_{g-r+1}=\Dd^\vee$  of a general $\h{D}\in\h{G}$ is birationally very ample: Recall that elements in ${\h{F}_{k,\delta}}^\vee$ are compounded,  thus
 $\h{F}_{k,\delta}\subsetneq\Gg$. 
 By  the well-known dimension estimate of a component of $\Gg^2_e$ whose general point is not compounded \cite[Theorem 10.1, pp. 845-846]{ACGH2}, 
 $$\dim\h{F}_{k,\delta}=3g-4-r\lneq\dim\Gg\le 3(g-r+1)+g-9 = 4g-3r-6,$$
 which is not compatible with the assumption $g\le 2r+2$.

 \vni 
 (b)
  Suppose the  residual series $\h{E}=g^2_{g-r+1}=\Dd^\vee$  of a general $\h{D}\in\h{G}$ is compounded 
  inducing an $n$-sheeted covering $C\stackrel{\pi}\rightarrow E$ onto an irrational curve $E$ of genus $h\ge 1$:
  Let $\Xx_{n,h}\subset\Mm_g$ be the Hurwitz space which is a closed subscheme of $\Mm_g$ consisting of degree $n$ ramified coverings of curves of genus $h$.  It is known that $\Xx_{n,h}$ is of pure dimension $2g+(2n-3)(1-h)-2$; cf. \cite[Th. 8.23, p. 828]{ACGH2}. By the assumption, there is a rational map $\eta: \Gg\to \Gg^\vee \dasharrow \Xx_{n,h}$. 
Recall that  by  a theorem of H. Martens \cite[Theorem 5.16]{EH2}, $$\dim W^r_d(C)\le d-2r-1$$ for any non-hyperelliptic curve $C$ of genus $g$ if  $g-d+r\ge 2$. For general $[C]\in\eta(\Gg)$, $$\dim\eta^{-1}([C])\le\dim G^r_{g+r-3}(C)=\dim W^r_{g+r-3}(C)\le g-r-4.$$  
 Thus we have 
 \begin{align*}
 3g-4-r&=\dim\h{F}_{k,\delta}\le\dim\Gg\le \dim\Xx_{n,h}+\dim\eta^{-1}([C])\\&\le(2g-2)+(g-r-4)=3g-6-r,
 \end{align*}
 a contradiction.
 \vni 
(c)
  Suppose that the  moving part $\Ee'$ of the residual series $\h{E}=g^2_{g-r+1}$  of a general $\h{D}\in\h{G}$
  induces a multiple covering onto a rational curve. Let $C\stackrel{\eta}\rightarrow E\subset\PP^2$
  be the morphism induced by $\Ee'$, $\deg\eta=k$.
  Set 
  $$\Psi:=\text{Bs}(\Ee), \quad\psi:=\deg\Psi 
  \quad\quad~\& ~\quad
  \Ee' :=|\Ee-\Psi|.$$ 
  Note  $\Ee'$ is a complete net (since $\Ee$ is), $\Ee'=\eta^*(\Oo_{\PP^2}(1))=\eta^*(\Oo_{E}(1))$,  thus 
  $$\deg E=2,   \quad \deg\Ee'=(g-r+1)-\psi=2k.$$ 
  Thus the residual series $\Ee=\Dd^\vee$
  of a general $\Dd\in\Gg$ is of the form $$\Ee=|\Ee'+\Psi|=|2g^1_k+\Psi|,$$where $g^1_k=|\eta^*(p)|, p\in E$. Hence
 $$\Dd=|K_C-2g^1_k-\Psi| \text{ for some }3\le k\le \left[\frac{g-r+1}{2}\right].$$ If $k=3$, $C$ is trigonal, $\Ee=|2g^1_3+\Psi|$, $\Psi\neq\emptyset$. However, this does not occur by Lemma \ref{easylemma}(ii). 
   
\vni
{\bf Conclusion:} (i) We deduce that the irreducible locus $$\h{F}_{k,\delta}=(2\h{G}^1_k+\overset{o}{\Cc}_\delta)^\vee$$ is indeed dense in a component 
$\h{G}\subset\widetilde{\h{G}}_\h{L}$,
which
gives rise to a component $\h{H}_{k,\delta}\subset \h{H}^\h{L}_{g+r-3}$  of  dimension, 
$$\dim\h{M}^1_{g,k}+\delta+\dim\Aut(\PP^r)=3g-4-r+\dim\Aut(\PP^r)$$
an $\Aut(\PP^r)$-bundle over the irreducible locus $\h{F}_{k,\delta}$;
$$\h{H}_{k,\delta} \dasharrow \h{F}_{k,\delta}\dasharrow \h{M}^1_{g,k}\subset\h{M}_g.$$
\vni 
(ii) Arguments (a), (b), (c) {\it do not} assert that the components $\h{H}_{k,\delta}\subset \h{H}^\h{L}_{g+r-3}$,
 $4\le k\le [\frac{g-r+1}{2}]$
are {\it all the components of} $\HL{g+r-3,g,r}$. There may exist other components different from 
$\h{H}_{k,\delta}$. 
\end{proof}

\section{Components with less than the expected number of moduli}

The notion of a component of the Hilbert scheme with the expected number of moduli was first  introduced by Sernesi in \cite{Sernesi} where he identified extensive class of triples $(d,g,r)$ such that $\HO{d,g,r}$ has a component $\Hh$ with the expected 
number of moduli, i.e. 
$$\dim\mu(\Hh)=3g-3+\rho(d,g,r).$$
There have been improvements on the study of Hilbert schemes along this line by works of A. Lopez and others; cf. \cite{Angelo1}, \cite{Angelo2} and references therein. Needless to say, there are many Hilbert schemes with components larger than the expected number of moduli. However, as were indicated in \cite[page 53, first paragraph]{Sernesi},  \cite[page 3467, end of Introduction]{Angelo2},
the existence of a component of a Hilbert scheme having strictly less than the expected number of moduli has not been known {\it widely enough }to the people working on this area and 
have not appeared explicitly in the literature in the context of our primary concern; {\it the number of moduli} $=\dim\mu(\Hh)$. 
We exhibit some examples of Hilbert schemes with components having the number of moduli less than expected,
mainly based on the discussion in previous sections. 

\subsection{Hilbert scheme with a component having less than the expected number of Moduli}
We exhibit the following examples of reducible Hilbert schemes with  components  having the moduli number less than expected. We retain notation used in previous section. 
\begin{example}\label{kgonal} (i) For $(d,g,r)=(17,14,6)$, $\HL{d,g,r}$
 has a component with the moduli number {\it less than} expected, dominating $\Mm^1_{g,4}$.
 \vni
 (ii) For $(d,g,r)=(31,21,13)$, $\HL{d,g,r}$ has a component with the moduli number {\it less than} expected, dominating a component of the 
 Hurwitz scheme $\Xx_{2,3}\subset\Mm_g$.
 \vni
 (iii) Let $(d,g,r)=(3r-1, 2r+2, r), r\ge 7, r\equiv1\pmod 2$.  
 
 \begin{enumerate}
\item[(iii-a)] For every $(k,\delta)$ such that 
 $$4\le k\lneqq\frac{r+3}{2}~~~~\hskip 8pt \& ~~~~\hskip 8pt\delta=g-r+1-2k>0,$$
 $\h{H}_{k,\delta}$ has {\it less than} the expected number of moduli. 
 
\item[(iii-b)] For $(k,\delta)=(\frac{r+3}{2},0)$,  $\h{H}_{k,0}$ has the expected number of moduli.
 \end{enumerate}
 (iv) For $(d,g,r)=(3r-1, 2r+2, r), r\ge 8, r\equiv0\pmod 2$, 
 $\h{H}_{k,\delta}$ has {\it less than} the expected number of moduli for 
 every $(k,\delta)$ such that
  $$4\le k\lneqq\frac{r+3}{2}~~~~\hskip 8pt \& ~~~~\hskip 8pt\delta=g-r+1-2k>0.$$


\end{example}
\begin{proof} (i) 
We consider $(d,g,r)=(2r+5, r+8, r), r=6$,  $(k,\delta)=(4,1)$ in Example \ref{extensive}.  Let 
$\Hh_{4,1}$ be the component dominating $\Mm^1_{g,4}$. We have 
$$\dim\mu({\Hh_{4,1}})=\dim\Mm^1_{g,4}=2g+3=31<3g-3+\rho(d,g,r)=32.$$

\vni (ii) For $(d,g,r)=(31,21,13)$, let $\Gg\subset \widetilde{\Gg}\subset\Gg^r_d$ be a component. We consider $\Ff:=\Gg^\vee\subset\Gg^2_{9}$. The following cases are possible. 
\begin{itemize}
\item[(1)] A general $\Ee\in \Ff$ is base-point-free and birationally very ample.
The irreducible family $\Ff$ corresponds to the Severi variety $\Sigma_{9,\zeta}$ parametrizing plane curves of degree $9$ with $\zeta=7$ nodal singularities, 
\begin{align*}
\dim\Ff&=\dim\Sigma_{9,\zeta}-\dim\Aut(\PP^2)\\
&=\dim\PP(H^0(\PP^2,\Oo_{\PP^2}(9)))-\zeta-\dim\Aut(\PP^2)\\&=
3g-3+\rho(d,g,r)=39.
\end{align*}
Let $\Hh\subset\HL{d,g,r}$ be the component over $\Gg$.
We have 
$$\dim\mu(\Hh)=\dim\Ff=3g-3+\rho(d,g,r)$$ and $\Hh$ is a component having the expected 
number of moduli.

\item[(2)] A general $\Ee\in \Ff$ is base-point-free, compounded inducing $C\stackrel{\eta}{\to} E\subset\PP^2$,
$\deg E=3$, $E$ is elliptic. On a triple covering $C\stackrel{\eta}{\to} E$ of an elliptic curve, $|K_C-\eta^*(K_E)|$ is 
very ample; otherwise there is $p+q\in C_2$ such that $|\eta^*(K_E)+p+q|=g^3_{11}$, which is base-point-free and birationally very ample, hence $\pi(11,3)\le 20$, a contradiction.  Thus 
$$C\in\Xx_{3,1},~~ \dim\Gg\ge \dim\Xx_{3,1}=2g-2=40>\lambda(d,g,r) =39.$$
This irreducible family has moduli number {\it more than expected}.
\item[(3)] A general $\Ee\in \Ff$ has non-empty base locus $\Delta$, $\deg\Delta=1$. 
$|\Ee-\Delta|$ induces $C\to F\subset\PP^2 $, $\deg F=4$, $F$ is smooth of genus $h=3$. 
We claim that $|K_C-\pi^*(K_F)-\Delta|$ is very ample: Since $|K_C-\pi^*(K_F)|$ is very ample \cite[Proposition 4.4]{KM}, 
$\dim|\pi^*(K_F)+p+q|=\dim|\pi^*(K_F)|=2$ for any $p+q\in C_2$. Suppose $\dim|\pi^*(K_F)+p+q+\Delta|\ge 3$.
Then $|\pi^*(K_F)+p+q+\Delta|=g^3_{11}$ is base-point-free,  birationally very ample,  $g\le \pi(11,3)=20$
whence the claim.
Thus 
$$\dim\mu({\Gg})=\dim\Xx_{2,3}=2g-4<\lambda(d,g,r)=39,$$
and this irreducible family  has the moduli number {\it less than} expected. 
\item[(4)] A general $\Ee\in \Ff$ has non-empty base locus $\Delta$, $\deg\Delta=1$,
$|\Ee-\Delta|$ induces $C\to F\subset\PP^2 $, $\deg F=2$, $F$ is rational, $C$ is $4$-gonal. 
This is the case in Example \ref{extensive} (iv); $g=r+8$, $(k,\delta)=(4,1)$.  We have 
$$\dim\mu({\Gg})=2g+3>\lambda(d,g,r)=39,$$
thus this family has  {\it more than} the expected number of moduli.
\item[(5)] A general $\Ee\in \Ff$ has non-empty base locus $\Delta$, $\deg\Delta\ge 2$. This is not possible
by Lemma \ref{easylemma}. 
\end{itemize}
In order to finalize our discussion, we denote by $\Hh_i\subset\HO{d,g,r},$ ($\Gg_i\subset\widetilde{\Gg}\subset\Gg^r_d$, resp.) $i=1,2,3,4$ the irreducible families obtained  in the items (1) -- (4)  above. We show that these families are dense in the
components to which they belong.
\begin{itemize}
\item[(1)] A general $C\in\Hh_1$ is $7$-gonal; $C^\vee\in\Sigma_{9,\zeta}$ with a $g^1_7$ cut out by lines through a node, by \cite{Coppens0} we have $$\text{gon}(C)=\text{gon}(C^\vee)=7,  \dim\mu(\Hh_1)=\lambda(d,g,r), \,\dim(\Hh_1)=\Xx(d,g,r).$$
\item[(2)] A general $C\in\Hh_2$ is $6$-gonal; $[C]\in\Xx_{3,1}$  with a $g^1_6=\eta^*(g^1_2)$, is $6$-gonal by Castelnuov-Severi inequality. $$\text{gon}(C)=6,  \dim\mu(\Hh_2)=\dim\Xx_{3,1}=40>\lambda(d,g,r),\dim\Hh_2>\Xx(d,g,r).$$
\item[(3)] A general $C\in\Hh_3$ has a $g^1_6=\eta^*(g^1_3)$ is $6$-gonal by Castelnuovo-Severi inequality,
$$\text{gon}(C)=6, \dim\mu(\Hh_3)=\dim\Xx_{2,3}=38<\lambda(d,g,r), \dim\Hh_3=\Xx(d,g,r).$$
\item[(4)] A general $C\in\Hh_4$ is $4$-gonal, 
$$\dim\mu(\Hh_4)=\dim\Mm^1_{g,4}=45>\lambda(d,g,r), \dim\Hh_4>\Xx(d,g,r).$$
\end{itemize}
We have $$\text{gon}(C_1)>\text{gon}(C_2)=\text{gon}(C_3)>\text{gon}(C_4)$$
while $$\dim\Hh_1=\dim\Hh_3<\dim\Hh_2<\dim\Hh_4.$$
By lower semi-continuity of gonality and obvious dimension reason as well, 

$\Hh_1$ is not in the boundary of $\Hh_2, \Hh_3, \Hh_4$, 

$\Hh_2$ is not in the boundary of $\Hh_3,\Hh_4$, 

$\Hh_3$ is not in the boundary of $\Hh_4$. 

\vni
Finally we claim that $\Hh_3$ is not in the boundary of $\Hh_2$. To see this,  we remark that any smooth limit of $k_1$-fold covers of genus $h_1$ curves must be a $k_2$-fold
cover of a genus $h_2$ curve for some $k_2\leq k_1$ and $h_2 \leq h_1$; from the theory of admissible covers,  when branched covers degenerate, one can keep the branch points separated and allow the domain/target curves to become nodal.  In this process the (arithmetic) genus remains the same, but the domain/target curves may become reducible whereas each irreducible component may have smaller genus and the degree of the cover restricted to a component may possibly become smaller;  cf. \cite[3G]{H3}  and \cite{HM}. This finishes the proof of the claim.
\vni
(iii), (iv): Recall for $(d,g,r)=(3r-1, 2r+2, r)$, $\lambda(d,g,r)=5r+2$ \& 
$$\dim\mu(\Hh_{k,\delta})=\dim\Mm^1_{g,k}=2g+5-2k=4r-1+2k;$$ cf. \eqref{kk}. 
Hence $$\dim\mu(\Hh_{k,\delta})<\lambda(d,g,r)$$
if and only 
\[4\le k\lneq \frac{r+3}{2} ~~\quad\&~~\quad
\begin{cases}
r\ge 7, \quad r\equiv1\pmod 2\\
r\ge 8,  \quad r\equiv 0\pmod 2
\end{cases}
\]
\end{proof}

\begin{remark} (i) The author believes that there exist further examples of Hilbert schemes having a component with less than the expected number of moduli.

\vni
(ii) Examples  we exhibited in this section are based on the following;
\begin{itemize}
\item[(a)] We first identify Hilbert scheme with a component $\Hh\subset \HL{d,g,r}$, $\Hh\to \Gg\to \Mm_g$ such that 
the irreducible family $\Gg^\vee\subset\Gg^{g-d+r-1}_{2g-2-d}$  consists of compounded linear series.
\item[(b)] We next check if the members $\Ee\in\Gg^\vee$ are allowed to have non-empty base locus $\Delta$, $|\Ee'|=|\Ee-\Delta|$, and then maximize the degree of the
base locus $\Delta$ so that $|K_C-\Ee|=|K_C-\Ee'-\Delta|$ is still very ample. 
\item[(c)] In this way, $\Gg^\vee$ (and hence $\Gg$) dominates an irreducible locus in $\Mm_g$ which forms an irreducible family with {\it relatively small dimension} consisting of the class of curves induced by the base point free part
$\Ee'$ of $\Ee\in\Gg^\vee$. 
\item[(d)] On the other hand, due
to the existence of base locus with positive degree, $\Gg^\vee$ is big enough to form a component but with smaller number
of moduli and hence the same thing holds for $\Hh\to\Gg\to\Mm_g$.
\end{itemize}
\end{remark}

\begin{remark}\label{discussion3} (i) There exist several irreducible or reducible Hilbert schemes $\HO{d,g,r}$ with all the  components
having more than the expected number of moduli. Such examples are not so rare, e.g.  the Hilbert scheme of 
Castelnuovo curves \cite{CC} or \cite[Proposition 5.1]{rmi}. 

\vni
(ii) There also exist  reducible Hilbert schemes $\HO{d,g,r}$ such that all the components have the expected number of moduli, e.g. $\HO{9,10,3}$ and many others. 

\vni
(iii) Furthermore, there exist  reducible Hilbert schemes $\HO{d,g,r}$ of partially mixed type, i.e. some component have the expected number of moduli and some other components have more than the expected number of moduli. For example,  $\HO{d,g,r}=\HO{32,21,14}$ has at least three components $\Hh_i, i=1,2,3$ such that:
\begin{itemize}
\item[(1)] $\Hh_1$ is the family of smooth plane curves of degree $8$ embedded in $\PP^{14}$ by $\Oo_{\PP^2}(4)$. $\Hh_1$ has the expected number of moduli.
\item[(2)] There is a component $\Hh_2$ dominating a component of the Hurwitz scheme $\Xx_{2,3}$, embedded in $\PP^{14}$ by $|K_C-\pi^*(K_E)|$, where $C\stackrel{\pi}{\to} E$, $\deg \pi=2$, $\text{genus}(E)=3$.
$\Hh_2$ has more than the expected number of moduli.
\item[(3)] There is a component $\Hh_3$  dominating $\Mm^1_{g,4}$.
$\Hh_3$ has more than the expected number of moduli.
\item[(4)] This is the one $\HL{2r+4,r+7,r}$ ($r=14$ case) which we encountered in Example \ref{first}. We just added one more component consisting of smooth plane curves of degree $8$ embedded in $\PP^{14}$.
\end{itemize}
\vni
(iv) In Example \ref{kgonal} (ii), we demonstrated a typical example of a reducible Hilbert scheme $\HL{d,g,r}=\HL{31,21,13}$ carrying 

a component having 
{\it less than expected moduli number} together with 

a component having the {\it expected moduli number} plus 

a component with {\it more than the expected moduli number}. 
\end{remark}
\begin{question}\label{discussion4}
\begin{itemize}
\item[(i)]
As far as the author knows, all the examples of a Hilbert scheme $\HO{d,g,r}$ with a component having less than the expected number of moduli, $\HO{d,g,r}$ also carries a component with or more than the expected number of moduli. 
\item[(ii)]It would be interesting to know if there is a  Hilbert scheme $\HO{d,g,r}$ (irreducible or irreducible) such that  {\it all} the components have less than the expected number of moduli. 
\end{itemize}
\end{question}

\bibliographystyle{spmpsci} 

\begin{thebibliography}{111}
\bibitem{AC2}
{E. Arbarello and M. Cornalba},
\textit{A few remarks about the variety of irreducible plane curves of given degree and genus.} Ann. Sei. \'Ecole Norm. Sup. (4) \textbf{16} (1983), 467--483.
\bibitem{ACGH}
{E. Arbarello, M. Cornalba, P. Griffiths, and J. Harris},
\textit{Geometry of Algebraic Curves Vol.I.}
Springer-Verlag, Berlin/Heidelberg/New York/Tokyo, 1985.
\bibitem{ACGH2}
{E. Arbarello, M. Cornalba, P. Griffiths, and J. Harris},
\textit{Geometry of Algebraic Curves Vol.II.}
Springer, 2011.
\bibitem{BFK}
{E. Ballico, C. Fontanari and C. Keem},
\textit{On the Hilbert scheme of linearly normal curves in $\mathbb{P}^r$ of relatively high degree.} J. Pure Appl. Algebra \textbf{ 224} (2020), no. 3, 1115--1123.
\bibitem{flood}
{E. Ballico,  C. Keem}, 
\textit{Hilbert scheme and Hilbert functions of smooth
curves of degrees at most $15$ in $\PP^5$.}
Preprint, 
available at http://arxiv.org/abs/2310.00682.
\bibitem{rmi}
{E. Ballico,  C. Keem}, 
\textit{On the Hilbert scheme of smooth curves of degree $15$ in $\PP^5$.}
To appear in \textsc{Proc. R. Soc. Edinb. A: Math.}, 
available at http://arxiv.org/abs/2310.00682.
\bibitem{CKP}
{K. Cho, C. Keem, and S. Park},
\textit{On the Hilbert scheme of trigonal curves and nearly extremal curves.}
Kyushu J. Math. \textbf{55} (2001), no. 1, 1–12.
\bibitem{CC}
{C. Ciliberto},
\textit{On the Hilbert Scheme of Curves of Maximal Genus in a Projective Space.} Mathematische Zeitschrift \textbf{194} (1987), 451--463.
\bibitem{Coppens0}
{M. Coppens},
\textit{The gonality of general smooth curves with a prescribed plane nodal model.} Math. Ann. \textbf{289} (1991), 89--93.
\bibitem{CKM}{M. Coppens, C. Keem and G. Martens}, \textit{The primitive length of a general $k$-gonal curve.} Indag. Math. (N.S.), \textbf{5} (1994), no. 2, 145--159.
\bibitem{CS}
{C. Ciliberto and E. Sernesi}, {\it Families of varieties and the Hilbert scheme.}
Lectures on Riemann surfaces (Trieste, 1987), 428-499, World Sci. Publ., Teaneck, NJ, 1989
\bibitem{H0}
{J. Harris},
\textit{Brill-Noether theory.}
Surveys in Differential Geometry
\textbf{14}(2009), 131--144; https://dx.doi.org/10.4310/SDG.2009.v14.n1.a5
\bibitem{EH2}
{D. Eisenbud and J. Harris},
\textit{The Practice of Algebraic Curves,
A Second Course in Algebraic Geometry.} American Mathematical Society, GSM 250, 2024.
\bibitem{H1}
{J. Harris},
\textit{Curves in Projective space.}
in ``Sem.Math.Sup.,", Press Univ. Montr\'eal, Montr\'eal, 1982.
\bibitem{H3}
{J. Harris and I. Morrison},
\textit{Moduli of curves.}
Springer-Verlag, Berlin/Heidelberg/New York, 1991.
\bibitem{HM}
{J. Harris and D. Mumford},
\textit{On the Kodaira dimension of the moduli space of curves}.
Invent. Math., \textbf{67} (1982), 23--88
\bibitem{Hartshorne}
{R. Hartshorne},
\textit{Algebraic Geometry.}
Springer-Verlag, Berlin/Heidelberg/New York, 1977.
\bibitem{Keem}
{C. Keem}, 
\textit{Reducible Hilbert scheme of smooth curves with positive Brill-Noether number.}
Proc. Amer. Math. Soc., \textbf{122} (1994), no. 2, 349--354.
\bibitem{lengthy}
{C. Keem},
\textit{Existence and the reducibility of the Hilbert scheme of linearly normal
  curves in $\mathbb{P}^r$ of relatively high degrees.} J. Pure Appl. Algebra \textbf{227} (2023), 1115--1123, available at
https://arxiv.org/abs/2101.00559.
\bibitem{KM}
{C. Keem and G. Martens}
\textit{Curves without plane model of small degree.}
Math. Nachr. 
\textbf{281} (2008), no. 12, 1791 – 1798.
\bibitem{KKL}
{C. Keem, Y.-H. Kim and A.F. Lopez},
\textit{Irreducibility and components rigid in moduli of the Hilbert Scheme of smooth curves.} Math. Z., \textbf{292} (2019), no. 3-4, 1207-1222.
\bibitem{Angelo1}{A. Lopez},
\textit{On the existence of components of the Hilbert scheme
with the expected number of moduli.}
Math. Ann.,
\textbf{289} (1991), 517-528.
\bibitem{Angelo2} {A. Lopez},
\textit{On the existence of components of the Hilbert scheme
with the expected number of moduli II.}
COMMUNICATIONS IN ALGEBRA, 
\textbf{27}(7) (1999), 3485-3493.
\bibitem{Sernesi}{E. Sernesi}, \textit{On the existence of certain families of curves.} Invent. Math., \textbf{75} (1984), no. 1, 25--57.
\bibitem{Sev1}
{F.  Severi}, 
\textit{Sulla classificazione delle curve algebriche e sul teorema di esistenza di Riemann.}
Rend. R. Acc. Naz. Lincei, \textbf{241} (1915), 
877--888.
\bibitem{Sev}
{F.  Severi}, 
\textit{Vorlesungen \"uber algebraische Geometrie.}
Teubner, Leipzig, 1921.
\end{thebibliography}

\end{document}